\documentclass{article}

\usepackage{tikz}
\usepackage{amssymb,amsthm,latexsym,amsmath,epsfig,pgf}

\usepackage{geometry}
\usepackage{doc}

\usepackage{url}

\usepackage{graphicx}
\usepackage{epstopdf}
\title{On Clique Roots of Flat Graphs}
\author{Hossein Teimoori Faal
\\
	Department of Mathematics and Computer Science,\\
Allameh Tabataba'i University, Tehran, Iran\\
}
\date{17 December 2021}

\newtheorem{theorem}{Theorem}[section]
\newtheorem*{theorem A}{Theorem A}
\newtheorem*{theorem B}{N\"olker's Theorem}
\newtheorem{lemma}{Lemma}[section]

\newtheorem{definition}{Definition}[section]
\newtheorem{example}{Example}[section]

\newtheorem{question}{Question}[section]
\newtheorem{conjecture}{Conjecture}
\theoremstyle{remark}

\theoremstyle{remark}

\begin{document}

\maketitle

\abstract{

A complete subgraph of a given graph is called a clique.
A clique Polynomial of a graph is a generating 
function of the number of cliques in $G$. 
A real root of the clique polynomial of a graph $G$ is called a 
\emph{clique root} of $G$. 
\\
Hajiabolhassan and Mehrabadi showed that the clique polynomial 
of any simple graph has a clique root in $[-1,0)$. 
As a generalization  of their result, 
the author of this paper showed that the class of $K_{4}$-free 
connected chordal graphs has also only clique roots. 
\\
A given graph $G$ is called flat if each edge of $G$ belongs to at most two triangles of $G$.  
In answering the author's open 
question about the class of \emph{non-chordal} graphs with 
the same property of having only c;ique roots, we extend the aforementioned result to the class of  
$K_{4}$-free flat graphs. In particular, we prove that the class of $K_{4}$-free flat graphs 
without isolated edges has $r=-1$ as one of it's clique roots. 
We finally present some interesting open questions and conjectures regarding 
clique roots of graphs. 
}





%
%
\section{Introduction}
\label{introduction}

Throughout this paper, we will assume that our graphs are finite, simple and undirected. 
For the definitions which are not mentioned here, readers can refer to the book \cite{BondyMurty82}. 
\\ 
For a given graph $G$, a complete subgraph of $G$ with $i$-vertices ($i \geq 1$)
is called an $i$-clique of $G$. 
The number of $i$-cliques of $G$ will be denoted by $c_{i}(G)$. By convention, 
we define $c_{0}(G)=1$ for any graph $G$. 
The largest clique of $G$ is called 
the \emph{clique number} of $G$ and is denoted by $\omega(G)$. 
The Clique polynomial of a graph $G$ denoted by $C(G,x)$ 
is defined as follows. 

\begin{equation}
C(G,x) = \sum_{k=0}^{\omega(G)} c_{k}(G) x^{k},
\end{equation}
where by convention $c_{0}(G)$ is taken to be $1$ for any simple graph $G$. 
A real root of $C(G,x)$ is called a \emph{clique root} of $G$. 
\\
Hajiabolhassan and late Mehrabadi \cite{HajiMehrab98}
showed that the clique polynomial of any graph $G$ has a 
clique root in $[-1,0)$. Moreover, they showed that 
clique polynomial of triangle-free graphs has only
clique roots. 
In the same line of research the author this paper \cite{Teimoo2018}
showed that clique polynomial of $K_{4}$-free connected chordal graphs 
has also only real roots. 
\\
An interesting class of graphs which are not (necessarily) chordal is 
the class of flat graphs which are those graphs in  which each edge belongs to 
at most two triangles of them. 
In this paper, in answering to a question in the author's previous paper \cite{Teimoo2018} 
about the existence of a nonchordal graph which has the clique polynomial 
with only real roots, we prove that the class of 
flat graphs has only clique roots. 
In particular, we show that the class of flat graphs without isolated edges has always a clique root $r=-1$. 

\section{Basic Definitions and Notations}

In this section, we quickly review clique polynomials and their algebraic properties. 
We also introduce a new class of graphs which are triangulated but they are not (necessarily) chordal.
\\
For a graph $G=(V,E)$, by $G-v$ we mean a graph obtained from $G$ by deleting the vertex $v \in V$. 
We also denote by $G[N_{G}(v)]$ the graph \emph{induced} by the (open) neighborhood of the vertex $v$. 

\subsection{Clique Polynomials}

\begin{definition}
	[see \cite{HajiMehrab98}]	
	For a given graph $G$, the clique polynomial of $G$
	denoted by $C(G,x)$ is defined by
	\begin{equation}
	C(G,x) = 
	1 + \sum_{i=1}^{\omega(G)} c_{i}(G) x^{i}. 
	\end{equation}	
	
\end{definition}

\begin{example}
	The clique polynomial of the graph $G_{1}$ depicted in Figure\ref{fig:M1}, is the following
	$$
	C(G_{1},x) = 1 + 4x + 4x^{2} + x^{3}. 
	$$
	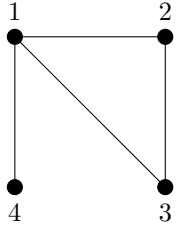
\begin{figure}
		\centering
		\begin{tikzpicture}
		\draw (0,0) -- (0,2) -- (2,2) -- (2,0);
		\draw (0,2)-- (2,0);
		\draw [fill] (0,0) circle [radius=0.1];
		\node [below] at (0,-0.1) {$4$};
		\draw [fill]  (2,0) circle [radius=0.1];
		\node [below] at (2,-0.1) {$3$};
		\draw [fill]  (0,2) circle [radius=0.1];
		\node [above] at (0,2.1) {$1$};
		\draw [fill]  (2,2) circle [radius=0.1];
		\node [above] at (2,2.1) {$2$};
		\end{tikzpicture}
		\caption{The clique polynomial of $G_{1}$ is $C(G_{1},x)= 1 + 4x + 4x^{2} + x^{3}$.}
		\label{fig:M1}
	\end{figure}
	
\end{example}

The following vertex-recurrence relation \cite{HajiMehrab98} can be obtained for a clique polynomial of a graph, 
using a simple counting argument. 

\begin{equation}\label{VerRecu1}
C(G,x)   
= C(G - v , x) + x C(G[N(v)], x) \hspace{0.5cm} (v \in V(G)).
\end{equation}
We also recall the following interesting \emph{vertex-deck} identity for the cliques
of $G$. 
\begin{lemma}
	Let $G=(V,E)$ be a graph on $n$ vertices. Then for any $i\geq 0$, we have
	\begin{equation}\label{CliqDeck1}
	(n-i)c_{i}(G) = 
	\sum_{v \in V}c_{i}(G-v)
	\end{equation}
\end{lemma}
\begin{proof}
	The proof is the straight-forward double-counting argument and left to the interested readers as a simple 
	exercise. 
\end{proof}

Using the vertex-recurrence relation (\ref{VerRecu1}) and the vertex-deck identity (\ref{CliqDeck1}), one can easily prove the following combinatorial 
interpretation of the first derivative of the clique polynomial \cite{LiGut95}. 
For the sake of completeness, we also include the proof. 
\begin{lemma}\label{Derivat1}
	For a graph $G=(V,E)$, we have 
	\begin{equation}\label{derivat1}
	\frac{d}{dx}C(G,x) = \sum_{v \in V}C(G[N(v)], x ).  
	\end{equation}
\end{lemma}

\begin{proof}

	By multiplying both sides of (\ref{CliqDeck1}) by 
	$x^{i}$ and summing over all $i$, we get
	
	\begin{equation}
	\sum_{i \geq 0}(n - i) c_{i}(G)x^{i}
	= \sum_{i \geq 0}\sum_{v \in V}c_{i}(G - v)x^{i},\nonumber
	\end{equation}
	or equivalently (by changing summation order), we have 
	\begin{equation}\label{equtt3}
	n\sum_{i \geq 0}c_{i}(G)x^{i} - x \sum_{i \geq 1}ic_{i}(G)x^{i-1} 
	= \sum_{v \in V}\left( \sum_{i \geq 0}c_{i}(G - v)x^{i} \right). 
	\end{equation}
On the other hand, 
by the definition of the first derivative of any polynomial we obtain
	
\begin{equation}\label{equtt4}
	\frac{d}{dx}C(G,x) = \sum_{i \geq 1}ic_{i}(G)x^{i-1} 
	. 
\end{equation}

Hence, considering the relation (\ref{equtt4}), we can rewrite the equality
	(\ref{equtt3})
	as follow
	\begin{equation}
	n . C(G,x) - x \frac{d}{dx}C(G,x)   
	= \sum_{v \in V}C(G - v, x)\nonumber. 
	\end{equation}
	or equivalently, 
	\begin{equation}\label{equtt5}
	\frac{d}{dx}C(G,x)   
	= \sum_{v \in V}\frac{C(G , x) - C(G - v, x)}{x}.
	\end{equation}
Finally based on the relations (\ref{equtt5})
and vertex-recurrence formula \ref{VerRecu1}, 
we get the desired result. 
	
\end{proof}

\begin{proof}
	By multiplying both sides of (\ref{CliqDeck1}) by 
	$x^{i}$ and summing over all $i$, we get
	
	\begin{equation}
	\sum_{i \geq 0}(n - i) c_{i}(G)x^{i}
	= \sum_{i \geq 0}\sum_{v \in V}c_{i}(G - v)x^{i},\nonumber
	\end{equation}
	or equivalently (by changing summation order), we have 
	\begin{equation}\label{equt3}
	n\sum_{i \geq 0}c_{i}(G)x^{i} - x \sum_{i \geq 1}ic_{i}(G)x^{i-1} 
	= \sum_{v \in V}\left( \sum_{i \geq 0}c_{i}(G - v)x^{i} \right). 
	\end{equation}
	On the other hand, 
	by the definition of the first derivative of any polynomial we obtain
	
	\begin{equation}\label{equt4}
	\frac{d}{dx}C(G,x) = \sum_{i \geq 1}ic_{i}(G)x^{i-1} 
	. 
	\end{equation}
	Hence, considering the relation (\ref{equt4}), we can rewrite the equality
	(\ref{equt3})
	as follow
	\begin{equation}
	n . C(G,x) - x \frac{d}{dx}C(G,x)   
	= \sum_{v \in V}C(G - v, x)\nonumber. 
	\end{equation}
	or equivalently, 
	\begin{equation}\label{equt5}
	\frac{d}{dx}C(G,x)   
	= \sum_{v \in V}\frac{C(G , x) - C(G - v, x)}{x}.
	\end{equation}
	Finally based on the relations (\ref{equt5})
	and vertex-recurrence formula \ref{VerRecu1}, 
	we get the desired result. 
	
\end{proof}

\subsection{Flat Graphs}

Next, we introduce a new class of graphs which can be of special interest for our purposes. They are 
triangulated graphs but they are not (necessarily) chordal. 
We call them flat graphs \cite{HaxelKostThomasse2012}.

\begin{definition}
	
	A graph $G$ is called \emph{flat} if every edge of $G$ belongs to at most 
	two triangles of $G$. 
	
\end{definition}

\begin{example}
	
	The wheel $W_{5}$ is a simple example of a flat graph which is clearly not a chordal graph. 
	
\end{example}

For the sequel purposes, we need the following generalization of the concept 
of a \emph{degree} of a vertex. 

\begin{definition}
	[see \cite{Teimoo2020}]	
	Let $G=(V,E)$ be a 
	graph and $e=\{u,v\} \in E$ be an edge of $G$. The value of $e$ in $G$ denoted by $val_{G}(e)$ is defined as follows:
	$$
	val_{G}(e) = N_{G}(u) \cap N_{G}(v). 
	$$
	An edge $e \in E$ with $val_{G}(e) = 0$ is called an \emph{isolated} edge of $G$. 
	
\end{definition}

\section{Main Results}

In  this section, we are going to prove the similar results of the author's paper
\cite{Teimoo2018}. The basic idea is that if the derivative of a cubic polynomial $p(x)$ has (at least) one real root,
then it has only real roots. 
Before going through the proof, we recall the well-known Euler identity for 
\emph{planar} graphs.

\begin{lemma}\label{Euler1}
	For a planar graph $G$ with $n$ vertices, $m$ edges and $t$ faces, we have 
	\begin{equation}
	n-m+f=2.
	\end{equation}
	In particular, for triangulated graph $f=t+1$, where $t$ denotes the number of triangles. 
\end{lemma}

We first state a simpler version of our main result. From now on, we 
will call a real root of $C(G,x)$ a \emph{clique root} of $G$. 

\begin{theorem}\label{Mainthm1}
	The class of $K_{4}$-free flat graphs without isolated edges 
	has only clique roots. In particular, this 
	class has always the clique root $r=-1$. 
\end{theorem}

\begin{proof}
	Since $G$ is $K_{4}$-free connected graph, we have
	\[
	C(G,x) = 1 + nx + mx^{2} + tx^3. 
	\]
	Now, we note that by the definition of flat graphs, the graph $G$ can not contain 
	any copy of $K_{3,3}$ as a subgraph (or as a minor in general). Hence by the Kuratowski's well-known theorem, 
	$G$ is a planar graph. 
	Therefore, based on the Euler identity\ref{Euler1}, we conclude that
	$C(G,-1)=0$. Hence, to complete the proof we only need to show that the polynomial 
	$
	\frac{d}{dx}C(G,x)
	$
	has (at least) one real root. 
	Next, we note that for a flat graph without isolated edges, the neighborhood of 
	any vertex $v \in V(G)$ is a path or a cycle. Let $\{v_{i}\}^{k_1}_{i=1}$ be the set of vertices
	such that each $G[N_{G}(v_{i})]$ is a path of length $l_{i}$ and $\{u_{j}\}^{k_2}_{j=1}$ is the set 
	of vertices such that each  $G[N_{G}(u_{j})]$ is a cycle of length $c_{j}$. 
	\\
	Now, considering Proposition \ref{Derivat1}, we get 
	\begin{eqnarray}\label{Derivat2}
	\frac{d}{dx}C(G,x) & = & 
	\sum_{v \in V}C(G[N_{G}(v)],x), \nonumber\\
	& = & \sum^{k_{1}}_{i=1}C(G[N_{G}(v_{i})],x) + \sum^{k_{2}}_{j=1}C(G[N_{G}(v_{j}),x]), \nonumber\\
	& = &  \sum^{k_{1}}_{i=1}(1+x)(1+ l_{i} x) + \sum^{k_{2}}_{j=1} (1+ c_{j} x + c_{j} x^{2} ). \nonumber 
	\end{eqnarray}
	Considering the fact that each $l_{i} \geq 2$ and $c_{j} \geq 4$, we immediately conclude that 
	\begin{equation}\label{keyineq1}
	\frac{d}{dx}C(G,-\frac{1}{2}) \leq 0.
	\end{equation}
	
	Finally, the result is immediate considering the fact that $C(G,0) > 0$, the equality \ref{keyineq1} and the \emph{intermediate value theorem} from calculus. 	
\end{proof}

Next, based on the vertex-recurrence formula 
\ref{VerRecu1} and some slight modifications of Theorem \ref{Mainthm1}, we can get the following interesting result. 

\begin{theorem}\label{Mainthm2}
	The class of $K_{4}$-free flat graphs in which no isolated edge lies in any induced cycle of $G$ 
	has only clique roots. In particular, this 
	class has always a clique root $r=-1$.
\end{theorem}

Now, we are in position to state and prove the main result of this paper which is stronger than the  
previous results. 

\begin{theorem}\label{MainThm}
	Let $G$ be a $K_{4}$-free flat graph. Then, the graph $G$
	has only clique roots. 
\end{theorem}

\begin{proof}
	We partition the vertex set $V(G)$ 
	into three (disjoint) sets $V_{1}$, $V_{2}$ and $V_{3}$. 
	$V_{1}$ is the set vertices $v$ such that $G[N_{G}(v)]$ is 
	a path. $V_{2}$ is the set of those vertices $u$ such that $G[N_{G}(u)]$
	is a cycle. $V_{3}$ is the set of vertices $w$ such that 
	$G[N_{G}(w)]$ is a collection of paths and isolated vertices (paths of length zero). 
	\\
	In this case, we do not have necessarily a clique root $r=-1$ but we still have a real 
	root because $C(G,x)$ is a cubic polynomial. Thus, again we need to prove that 
	$
	\frac{d}{dx}C(G,x)
	$ 
	has (at least) one real root. 
	\begin{eqnarray}\label{Derivat3}
	\frac{d}{dx}C(G,x) & = & 
	\sum_{v \in V}C(G[N_{G}(v)],x), \nonumber\\
	& = & \sum_{v \in V_{1}}C(G[N_{G}(v)],x) + \sum_{u \in V_{2}}C(G[N_{G}(u)],x) + \sum_{w \in V_{3}}C(G[N_{G}(w)],x) , \nonumber\\
	& = &  \sum_{v \in V_{1}}(1+x)(1+ l_{v} x) + \sum_{u \in V_{2}} (1+ c_{u} x + c_{u} x^{2} ) \nonumber \\
	& + & \sum_{w \in V_{3}}\Big((1+x)(1+l_{w}x) + (1+k_{w}x) -k_{w}\Big), \nonumber 
	\end{eqnarray}
	where the last line is based on the clique polynomial formula for the disjoint union of two graphs. 
\end{proof}

Once again, considering the fact that each $l_{i} \geq 2$, $c_{j} \geq 4$ and $k_{w} \geq 1$ 
it is not hard to prove that 
$$
\frac{d}{dx}C(G,-\frac{1}{2}) \leq 0,
$$
which again by the intermediate value theorem implies that $\frac{d}{dx}C(G,x)$ has at least one real root. 
This finishes the proof.

\section{Open Questions and Conjectures}

In this final section, we propose several interesting 
open questions and conjectures regarding the clique roots of graphs. 
\\
In the paper \cite{Teimoo2018}, the author gives a proof of the real rootedness of 
the clique polynomial of $K_{4}$-free chordal graphs mainly 
based on the idea of the \emph{breadth first search} algorithm. 
Therefore, the following question is of our special interest. 

\begin{question}
	
	Can we give an algorithmic proof of Theorem \ref{MainThm}?

\end{question}

It is also interesting to know that whether Theorem \ref{MainThm} is true
for the class of $K_{5}$-free graphs or not. Here is a counter-example. 
Let $G$ be a the graph $K^{+}_{4}$ which is a complete graph with an isolated edge 
attached to one of its vertices. Hence, we have 
$$
C(K^{+}_{4},x) = (1 + x)^{4} + (1+x)^{2} - 1. 
$$  
Therefore, 
$
\frac{d}{dx}C(K^{+}_{4},x) = 4(1+x)^{3} + 2(1+x)
$
and this implies that
$
\frac{d^{2}}{dx^{2}}C(K^{+}_{4},x) = 12(1+x)^{2} + 2 >0.
$
Thus, we immediately conclude that not all roots of 
$
C(K^{+}_{4}, x)
$
are real. 
The following open question naturally arises. 

\begin{question} 
	Which subclasses of flat graphs have only clique roots? 
\end{question} 

We first note that a generalization of the edge value of an edge $e \in E(G)$
of a graph $G$, which we call it \emph{triangle value} \cite{Teimoo2020} can be defined as follows. 

\begin{definition}
	For a triangle $\delta = \{u,v,w\} \subset V(G)$,
	the value of $\delta$ in $G$ denoted by $val_{G}(\delta)$ is defined by 
	$$
	val_{G}(\delta) = N_{G}(u) \cap N_{G}(v) \cap N_{G}(w). 
	$$
	Moreover, a triangle is called \emph{isolated} if its triangle value is zero. 
\end{definition} 

In analogy with Theorem \ref{Mainthm1}, we made the following conjecture. 

\begin{conjecture}

	The class of $2$-connected $K_{5}$-free flat graphs without isolated triangles 
	has only clique roots. In particular, this 
	class has always the clique root $r=-1$ of multiplicity $2$.
	
\end{conjecture}

A $i$-clique is called \emph{odd} clique if $i$ is odd and otherwise is called an \emph{even} clique.  
We can define an interesting parameter in graph theory as the 
absolute value of the difference between the total number of odd and even cliques. 
More precisely, we define 
\begin{equation}
D_{clique}(G) = 
\Big\vert 
\sum_{i:~i~is~odd}c_{i}(G) - \sum_{i:~i~is~even}c_{i}(G)
\Big\vert
. 
\end{equation}

As we can see from Theorem \ref{Mainthm1} and Theorem \ref{Mainthm2}, for the subclasses of flat graphs $G$ 
mentioned in those theorems, we conclude that
$
D_{clique}(G) = 1
$
. 
Thus, the following question naturally arises. 

\begin{question} 
	For which classes of graphs 
	there exits a positive integer $k$ such that 
	$
	D_{clique}(G) \leq k
	$
	?
\end{question} 

We note that when $k=1$, we call the graph $G$ a \emph{balanced} graph.
We believe that the following conjecture is true.

\begin{conjecture}
	Any flat graph is a balanced graph.	
	
\end{conjecture}


\end{document}